\documentclass[11pt]{amsart}
\usepackage{amsthm}
\usepackage{cite}
\usepackage{array}
\usepackage{amsmath}
\usepackage{enumerate}
\usepackage{tikz}
\usetikzlibrary{calc}
\usepackage{amssymb}
\usepackage{latexsym}
\usepackage{amsfonts}
\usepackage{color}
\usepackage{mathrsfs}
\usepackage{epsfig,helvet}

\newcommand{\Z}{{\mathbb Z}}

\theoremstyle{plain} \numberwithin{equation}{section}
\newtheorem{thm}{Theorem}[section]
\newtheorem{theorem}[thm]{Theorem}
\newtheorem{lemma}[thm]{Lemma}
\newtheorem{corollary}[thm]{Corollary}
\newtheorem{example}[thm]{Example}
\newtheorem{definition}[thm]{Definition}
\newtheorem{proposition}[thm]{Proposition}

\begin{document}
\setcounter{page}{1}

\title[ Hasan and Padhan]{Detecting capable pairs of some nilpotent Lie superalgebras}

	\author{IBRAHEM YAKZAN HASAN}
	\address{Centre for Data Science, Institute of Technical Education and Research \\
	Siksha `O' Anusandhan (A Deemed to be University)\\
	Bhubaneswar-751030 \\
	Odisha, India}
\email{ibrahemhasan898@gmail.com}
	\author[Padhan]{Rudra Narayan Padhan}
\address{Centre for Data Science, Institute of Technical Education and Research  \\
	Siksha `O' Anusandhan (A Deemed to be University)\\
	Bhubaneswar-751030 \\
	Odisha, India}
\email{rudra.padhan6@gmail.com, rudranarayanpadhan@soa.ac.in}

\author{Manjula Das}
	\address{ Siksha `O' Anusandhan (A Deemed to be University)\\
	Bhubaneswar-751030 \\
	Odisha, India}  
\email{manjuladas@soa.ac.in}

\subjclass[2010]{Primary 17B30; Secondary 17B05.}
\keywords{Heisenberg Lie superalgebra; multiplier; capability; non-abelian tensor and exterior product}
\maketitle

\begin{abstract}
In  this article, we define the capable pairs of Lie superalgebras. We classify all capable pairs of abelian and Heisenberg Lie superalgebras. After that we discuss on pairs of Lie superalgebras with derived subalgebra of dimension one and a non-abelian ideal. Finally, we determine the structure of the Schur multiplier of pairs of Heisenberg Lie superalgebras.
\end{abstract}

\section{Introduction}

In the twentieth century numerous studis have been conducted on the classification of $p$-groups. A group $G$ is said to be capable if there exists a group $K$  such that $G \cong K/Z(K)$. The notion of capable group was introduced by Baer \cite{Baer1938} and he classified all finitely generated capable abelian groups. Hall \cite{Hall}, pointed out that a capable group plays an important role in order to classify $p$-groups. Beyl et. al. \cite{Beyl1979} defined the epicenter $Z^{*}(G)$ of a group $G$, and they proved that a group $G$ is capable if and only if $Z^{*}(G)=1$. Further, exterior square of a group was studied for the first time in \cite{Brown2010}, which has an interesting relation with the capability of a group. Exterior center $Z^{\wedge}(G)$ of a group $G$ is defined as $Z^{\wedge}(G)=\{g \in G|~g\wedge h=1,~ \forall~h \in G\}$. Ellis \cite{Ellis1995} proved that $Z^{\wedge}(G)=Z^{*}(G)$. For more details on classification of capable abelian groups and capable extra-special $p$-groups one can go through \cite{Baer1938, Tappe,p10}. In 1996, Ellis \cite{Ellis1996} generalized the notion of capability of group to pair of groups $(G,N)$ where $N$ is a normal subgroup of $G$. He investigated capability, Schur multipliers, and central series for pairs of groups, and classified all finitely generated abelian capable pairs.

\smallskip

Analogous to the notion of capable group, for Lie algebra capability has been defined and a lot of inverstigation has been done \cite{org, Batten1993, Ellis1991, p1, Niroomand2011, Alamian2008, N2, Russo2011,PMF2013, BS1996, Batten1996, Ellis1987, Hardy1998, Hardy2005, Kar1987, Moneyhun1994}. Recently, capability of  pair of Lie algebras has been defined and classified for the capable pairs of abelian and Heisenberg Lie algebras \cite{org}.

\smallskip

Lie superalgebras have applications in many areas of Mathematics and Theoretical Physics as they can be used to describe supersymmetry. Kac \cite{Kac1977} gives a comprehensive description of mathematical theory of Lie superalgebras, and establishes the classification of all finite dimensional simple Lie superalgebras over an algebraically closed field of characteristic zero. In the last few years the theory of Lie superalgebras has evolved remarkably, obtaining many results in representation theory and classification. Most of the well-known results of Lie algebras can be extended to Lie superalgebras but all are not trivial. The notion of multiplier and cover for Lie algebras is generalized to the case of Lie superalgebras and studied in \cite{ SN2018a,Nayak2018, NPP}. Since the last five years, a lot of research work has been carried out on the classification of capable Lie superalgebras and Schur multiplier, for more details see \cite{SN2018b,Padhandetec,hesam, p50,p51,p52,p53}. The aim of this paper is to discuss the capability of pair of Lie superalgebras and classify the capable pairs of abelian and Heisenberg Lie superalgebras.

%In 1904, I. Schur introduced the Schur multiplier and cover of a group in his work on projective representation. Batten \cite{Batten1993} introduced and studied Schur multiplier and cover of a Lie algebra and later on, studied by several authors \cite{BS1996, Batten1996}. For a finite dimensional Lie algebra $L$ over a field $\mathbb{F}$ the free presentation of $L$ is the exact sequence, $0\longrightarrow R \longrightarrow F\longrightarrow L$, where $F$ is a free Lie algebra and $R$ is an ideal of $F$. Then the Schur multiplier $\mathcal{M}(L)$ is isomorphic to $F' \cap R/[F,R]$. Moneyhun \cite{Moneyhun1994} proved that for a finite dimensional Lie algebra $L$ of dimension $n$, $\dim \mathcal{M}(L)=\dfrac{1}{2}n(n-1)-t(L)$, where $t(L)\geq 0$. This bound is used to classify finite dimensional nilpotent Lie algebras $L$ with some small values of $t(L)$. Specifically, the complete classification of $L$ with $t(L)\leq 8$, has been depicted in \cite{Batten1996, Hardy1998, Hardy2005, Russo2011}. An improved bound for $\dim \mathcal{M}(L)$ for a non-abelian nilpotent Lie algebra $L$ is further given by Niroomand and Russo \cite{Niroomand2011}, and using this some classifications of $L$ are done with lesser effort.

\smallskip

 \smallskip
 
%Baer \cite{Baer1938} defined the notion of capable group. Beyl et. al.,  \cite{Beyl1979} introduced the epicenter $Z^{*}(G)$ of a group $G$, and they proved that a group $G$ is capable if and only if $Z^{*}(G)=1$. Further exterior square of a group was studied for the first time in \cite{Brown2010}, which has an interesting relation with the capability of a group. Exterior center $Z^{\wedge}(G)$ of a group $G$ is defined as $Z^{\wedge}(G)=\{g \in G|~g\wedge h=1,~ \forall~h \in G\}$. Ellis \cite{Ellis1995} proved that $Z^{\wedge}(G)=Z^{*}(G)$. Similarly the notion of epicenter $Z^{*}(L)$ of a Lie algebra $L$ is given by Alamian et. al., \cite{Alamian2008}. The non-abelian tensor product, and exterior product of Lie algebras are defined, and some of the properties are studied by Ellis \cite{Ellis1987, Ellis1991, Ellis1995}. Recently, Niroomand et. al., \cite{PMF2013} investigated the connection between epicenter and exterior center of a finite dimensional Lie algebra. Finally, they have classified all capable Heisenberg Lie algebras, and in continuation, as an application they have shown that there exists at least one capable Lie algebra of arbitrary co rank.

 \smallskip
 
 %The exterior product of Lie superalgebras, and some of its properties are discussed in the paper \cite{GKL2015}. In this paper we prove some important properties of non-abelian tensor product as well as exterior product of Lie superalgebras. Further we define capable Lie superalgebra. Then characterization of capable Lie superalgebras are given which are extended versions of results of Lie algebras  \cite{Alamian2008}. Following \cite {PMF2013}, first we classify all capable Heisenberg Lie superalgebras, and then all finite dimensional capable nilpotent Lie superalgebras with dimension of derived subalgebras at most one. We hope, this leads to the classification of non-abelian finite dimensional nilpotent Lie superalgebras $L$ with $s(L)=3$.

\section{Preliminaries and Auxiliary Results}

Let $\mathbb{Z}_{2}=\{\bar{0}, \bar{1}\}$ be a field. A $\mathbb{Z}_{2}$-graded vector space $V$ is simply a direct sum of vector spaces $V_{\bar{0}}$ and $V_{\bar{1}}$, i.e., $V = V_{\bar{0}} \oplus V_{\bar{1}}$. It is also referred as a superspace. We consider all vector superspaces and superalgebras are over field $\mathbb{F}$ (characteristic of $\mathbb{F} \neq 2,3$). Elements in $V_{\bar{0}}$ (resp. $V_{\bar{1}}$) are called even (resp. odd) elements. Non-zero elements of $V_{\bar{0}} \cup V_{\bar{1}}$ are called homogeneous elements. For a homogeneous element $v \in V_{\sigma}$, with $\sigma \in \mathbb{Z}_{2}$ we set $|v| = \sigma$ is the degree of $v$. A  subsuperspace (or, subspace)  $U$ of $V$ is a $\mathbb{Z}_2$-graded vector subspace where  $U= (V_{\bar{0}} \cap U) \oplus (V_{\bar{1}} \cap U)$. We adopt the convention that whenever the degree function appears in a formula, the corresponding elements are supposed to be homogeneous. 

\smallskip 

A  Lie superalgebra (see \cite{Kac1977, Musson2012}) is a superspace $L = L_{\bar{0}} \oplus L_{\bar{1}}$ with a bilinear mapping $ [., .] : L \times L \rightarrow L$ satisfying the following identities:
\begin{enumerate}
\item $[L_{\alpha}, L_{\beta}] \subset L_{\alpha+\beta}$, for $\alpha, \beta \in \mathbb{Z}_{2}$ ($\mathbb{Z}_{2}$-grading),
\item $[x, y] = -(-1)^{|x||y|} [y, x]$ (graded skew-symmetry),
\item $(-1)^{|x||z|} [x,[y, z]] + (-1)^{ |y| |x|} [y, [z, x]] + (-1)^{|z| |y|}[z,[ x, y]] = 0$ (graded Jacobi identity),
\end{enumerate}
for all $x, y, z \in L$. Clearly $L_{\bar{0}}$ is a Lie algebra, and $L_{\bar{1}}$ is a $L_{\bar{0}}$-module. If $L_{\bar{1}} = 0$, then $L$ is just a Lie algebra, but in general a Lie superalgebra is not a Lie algebra. A Lie superalgebra $L$, is called abelian if  $[x, y] = 0$ for all $x, y \in L$. Lie superalgebras without the even part, i.e., $L_{\bar{0}} = 0$, are  abelian. A subsuperalgebra (or subalgebra) of $L$ is a $\mathbb{Z}_{2}$-graded vector subspace which is closed under bracket operation. The graded subalgebra $[L, L]$,  of $L$  is known as the derived subalgebra of $L$. A $\mathbb{Z}_{2}$-graded subspace $I$ is a graded ideal of $L$ if $[I, L]\subseteq I$. The ideal 
\[Z(L) = \{z\in L : [z, x] = 0\;\mbox{for all}\;x\in L\}\] 
is a graded ideal and it is called the {\it center} of $L$. A homomorphism between superspaces $f: V \rightarrow W $ of degree $|f|\in \mathbb{Z}_{2}$, is a linear map satisfying $f(V_{\alpha})\subseteq W_{\alpha+|f|}$ for $\alpha \in \mathbb{Z}_{2}$. In particular, if $|f| = \bar{0}$, then the homomorphism $f$ is called homogeneous linear map of even degree. A Lie superalgebra homomorphism $f: L \rightarrow M$ is a  homogeneous linear map of even degree such that $f([x,y]) = [f(x), f(y)]$ for all $x, y \in L$.  If $I$ is an ideal of $L$, the quotient Lie superalgebra $L/I$ inherits a canonical Lie superalgebra structure such that the natural projection map becomes a homomorphism. The notions of {\it epimorphisms, isomorphisms} and {\it automorphisms} have the obvious meaning. 
\smallskip

Throughout this article for superdimension of Lie superalgebra $L$ we simply write $\dim L=(m\mid n)$, where $\dim L_{\bar{0}} = m$ and $\dim L_{\bar{1}} = n$. Also  $A(m \mid n)$ denotes an abelian Lie superalgebra where $\dim A=(m\mid n)$. A  Lie superalgebra $L$ is said to be Heisenberg Lie superalgebra if $Z(L)=L'$ and $\dim Z(L)=1$. According to the homogeneous generator of $Z(L)$, Heisenberg Lie superalgebras can further split into even or odd Heisenberg Lie superalgebras \cite{MC2011}. By Heisenberg Lie superalgebra we mean special Heisenberg Lie superalgebra in this article. Now we list some useful results from \cite{Nayak2018}, for further use.

\begin{theorem}\label{th3.3}\cite[See Theorem 3.4]{Nayak2018}
\[\dim \mathcal{M}(A(m \mid n)) = \big(\frac{1}{2}(m^2+n^2+n-m)\mid mn \big).\]
\end{theorem}

\begin{theorem} \label{th3.4}\cite[See Theorem 4.2, 4.3]{Nayak2018}
Every Heisenberg Lie superalgebra with even center has dimension $(2m+1 \mid n)$ and  is isomorphic to $H(m , n)=H_{\overline{0}}\oplus H_{\overline{1}}$, where
\[H_{\overline{0}}=<x_{1},\ldots,x_{2m},z \mid [x_{i},x_{m+i}]=z,\ i=1,\ldots,m>\]
and 
\[H_{\overline{1}}=<y_{1},\ldots,y_{n}\mid [y_{j}, y_{j}]=z,\  j=1,\ldots,n>.\]
Further,
$$
\dim \mathcal{M}(H(m , n))=
\begin{cases}
 (2m^{2}-m+n(n+1)/2-1 \mid 2mn)\quad \mbox{if}\;m+n\geq 2\\
(0 \mid 0) \quad \mbox{if}\;m=0, n=1\\  
   (2 \mid 0)\quad \mbox{if}\;m=1, n=0.
\end{cases}
$$
\end{theorem}

The following is the established result for the multiplier and cover of Heisenberg Lie superalgebra of odd center.

\begin{theorem}\label{th3.6}\cite[See Theorem 2.8]{SN2018b}
Every Heisenberg Lie superalgebra, with odd center has dimension $(m \mid m+1)$, is isomorphic to $H_{m}=H_{\overline{0}}\oplus H_{\overline{1}}$, where
\[H_{m}=<x_{1},\ldots,x_{m} , y_{1},\ldots,y_{m},z \mid [x_{j},y_{j}]=z,   j=1,\ldots,m>.\]
Further,
$$
\dim \mathcal{M}(H_{m})=
\begin{cases}
 (m^{2}\mid m^{2}-1)\quad \mbox{if}\;m\geq 2\\  
(1\mid 1) \quad \quad \mbox{if}\;m=1.  
  \end{cases}
$$
\end{theorem}

\begin{definition}
A Lie superalgebra $L$ is said to be $capable$ if there exists a Lie superalgebra $H$ such that $L \cong H/Z(H)$. 
\end{definition}

\begin{theorem}\label{thm555}\cite[Theorem 6.7]{Padhandetec} 
	$H_m$ is capable if and only if $m = 1$.

\end{theorem}
\begin{theorem}\label{th4.22}\cite[ Theorem 6.3]{Padhandetec}
	$A(m \mid n)$ is capable if and only if $m=0, n=1$ or $m+n \geq 2$.
\end{theorem}

\begin{theorem}\label{th4.33}\cite[ Theorem 6.4]{Padhandetec}  
	$H(m, n)$ is capable if and only if $m=1, n=0$.
\end{theorem}

The following theorem determines the structure of capable Lie superalgebra $L$ of dimension $(k \mid l)$ with $\dim L^2 =1$.

\begin{theorem}\label{th4.4}\cite[Proposition 3.4]{SN2018b}\cite[Theorem 6.9]{Padhandetec} \label{th5a}
Let $L$ be a nilpotent Lie superalgebra of dimension $(k \mid l)$ with $\dim L'=(r \mid s)$, where $r+s=1$. If $r=1, s=0$ then $L \cong H(m,n)\oplus A(k-2m-1 \mid l-n)$ for $m+n\geq 1$. If $r=0, s=1$ then $L \cong H_{m} \oplus A(k-m \mid l-m-1)$. Moreover, $L$ is capable if and only if either $L \cong H(1 , 0)\oplus A(k-3 \mid l)$ or $L \cong H_{1}\oplus A(k-1 \mid l-2)$.
\end{theorem}

Now we determine the structure of non-abelian graded ideals of a Heisenberg Lie superalgebra, which follows immediate from Theorem \ref{th4.4}.

\begin{lemma}\label{lemm78}
Let $I$ and $J$ be two ideals of a Heisenberg Lie superalgebras $ H(m,n)$ and $H_m$ respectively, with $\dim I=(k \mid h),~ 1\leq k \leq 2m+1,~ 1\leq h\leq n$ and $\dim J=(p \mid q),~ 1\leq p\leq m, ~1\leq q\leq m+1$. Then $I$ is Heisenberg or $I\cong H(r,s)\oplus A(k-2r-1\mid h-s) $ for some $1\leq r \leq m$ and $1\leq s \leq n$ . Similarly, $J$ is Heisenberg or $J\cong H_r\oplus A(p-r\mid q-r-1 ) $ for some $1\leq r \leq m$ .
\end{lemma}

Let $X=X_{\bar{0}} \cup X_{\bar{1}}$ be a totally ordered $\mathbb{Z}_{2}$-graded set. Let $\Gamma(X)$ be the groupoid of non-associative monomials in the alphabet $X$, $u \circ v=(u)(v)$ for $u,v \in \Gamma (X)$, and $S(X)$ be the free semigroup of associative words with the bracket removing homomorphism $-: \Gamma (X)\longrightarrow S(X) $. For $u=x_{1}\ldots x_{n} \in S(X),~x_{i} \in X$, we consider the word length $l_{X}(u)=n$, the multi degree $m(u)$, $|u|=\sum^{n} _{i=1} |x_{i}| \in \mathbb{Z}_{2}$. Now let $K$ be a commutative ring having identity element $1$, and let $A(X)$ and $F(X)$ be the free associative and non-associative $K$-algebras respectively. Let $A(X)_{\sigma}, F(X)_{\sigma}$  for $\sigma \in \mathbb{Z}_{2}$ be the $K$-linear spans of the subsets $S(X)_{\sigma}$ and $\Gamma(X)_{\sigma}$ respectively, $A(X)$ and $F(X)$ being the $\mathbb{Z}_{2}$-graded associative and non-associative algebras respectively. Let $I$ be the ideal generated by the homogeneous elements of the form $x\circ y - (-1)^{|x||y|} y\circ x$ and $(x\circ y)\circ z-x\circ (y \circ z)-(-1)^{|x||y|} y \circ (x\circ z)$, for $x, y \in \Gamma(X)$ then $L(X)=F(X)/I$, is the free Lie $K$-superalgebra over $X$ (see \cite{YAM2000}).

\smallskip

Suppose the set $S(X)$ is ordered lexicographically.
A monomial $u \in \Gamma (X)$ is said to be regular if either $u \in X$ or;
\begin{enumerate}
\item $u=u_{1}\circ u_{2}$ where $u_{1},~ u_{2}$ are regular monomials with $\overline{u_{1}}> \overline{u_{2}}$,
\item $u=(u_{1}\circ u_{2})\circ u_{3}$ with $\overline{u_{2}} \leq \overline{u_{3}}$. 
\end{enumerate}

A monomial $u \in \Gamma (X)$ is said to be $s$-regular if either $u$ is a regular monomial or $u=(v)(v)$ with $v$ a regular monomial and $|v|=1$. Then the set of all images of the s-regular monomials form a basis of the free Lie superalgebra $L(X)$. The analogue of Witt's formula can be seen in Corollary 2.8 in \cite{YAM2000}. 
\begin{theorem}\label{t11}
Let $X=X_{\bar{0}} \cup X_{\bar{1}}$, $X_{\bar{0}}=\{x_{1},\ldots,x_{m} \},~X_{\bar{1}}=\{y_{1},\ldots,y_{n} \}$ be a totally ordered $\mathbb{Z}_{2}$-graded set and $L(X)$ be the free Lie superalgebra, $\mu (l)$ the M\"{o}bius function, and $W(\alpha_{1},\ldots,\alpha_{m+n})$ the rank of the free module of elements of multi degree $\alpha=(\alpha_{1},\ldots,\alpha_{m+n})$ in the free Lie algebra of rank $m+n$, 
$$ W(\alpha_{1},\ldots,\alpha_{m+n})=~\dfrac{1}{|\alpha|}\sum _{e|\alpha}\mu(e)\dfrac{(|\alpha|/e)!}{(\alpha/e)!},$$
where $|\alpha|=\sum ^{m+n}_{i=1}\alpha_{i}$. Let $SW(\alpha_{1},\ldots,\alpha_{m+n})$ be the rank of the free module of elements of multi degree $\alpha=(\alpha_{1},\ldots,\alpha_{m+n})$ in the free Lie superalgebra $L(X)$ of rank $m+n$. Then 
$$ SW(\alpha_{1},\ldots,\alpha_{m+n})=W(\alpha_{1},\ldots,\alpha_{m+n})+~\beta W\left(\frac{\alpha_{1}}{2},\ldots,\frac{\alpha_{m+n}}{2}\right), $$
where
$$
\beta=
\begin{cases}
 0\;\;if~ there~exists~an~i~such ~that~ \alpha_{i}~ is~odd,~or~if~\frac{1}{2}\sum ^{m+n}_{i=m+1}\alpha_{i}|\alpha_{i}|=1, \\
1\;\; otherwise.
   \end{cases}
$$

\end{theorem}

\begin{theorem}\label{c12}\cite{pet2000}
	\begin{align*}
		\begin{split}
	& dim L_r=\frac{1}{r}\sum_{a|r}^{}\mu(a)(m-(-1)^an)^{r/a}\\
	& dim L_{r,+}=\frac{1}{r}\sum_{a|r}^{}\mu(a)\frac{(m-(-1)^an)^{r/a}+(m-n)^{r/a}}{2}\\
	& dim L_{r,-}=\frac{1}{r}\sum_{a|r}^{}\mu(a)\frac{(m-(-1)^an)^{r/a}-(m-n)^{r/a}}{2}\\
	& sdimL_r=\frac{1}{r}\sum_{a|r}^{}\mu(a)(m-n)^{r/a}\\
	& dimL_\alpha=\frac{(-1)^{|\alpha_-|}}{|\alpha|}\sum_{a|r}^{}\mu(a)\frac{(|\alpha|/a)!(-1)^{|\alpha_-|/a}}{(\alpha_1/a)!\cdots(\alpha_{m+n}/a)!}
	\end{split}
	\end{align*}

\end{theorem}

\section{ Exterior-Products of Lie Superalgebras}

Let $I$ be a graded ideal of a Lie superalgebra $L$. Then $(L,I)$ is called a pair of Lie superalgebras. In this section we recall some of the known notation and results from \cite{Padhandetec, GKL2015}. Recently, X. Garc\'{i}a-Mart\'{i}nez et al. \cite{GKL2015} introduced the notions of non-abelian tensor product of Lie superalgebras and exterior product of Lie superalgebras over $K$, where $K$ is a commutative ring. Here we will state some similar definition to \cite{GKL2015} for pair of Lie superalgebras.

 \begin{definition}
 Let $(L,I)$ be a pair of Lie superalgebras. A $\mathbb{F}$-bilinear map of even degree is said to be an action of $L$ on $I$ \[L\times I \longrightarrow I,~~~~~~~(l, i)\mapsto {}^ li, \]
 if 
 \begin{enumerate}
 \item ${}^{[l , l']} i = {}^l({}^{l'}i)-(-1)^{|l||l'|}~  {}^{l'}({}^{l}i),$
 \item ${}^l{[i , i']}=[{}^li, i']+(-1)^{|l||i|}[i , {}^li'],$
\end{enumerate}
for all homogeneous $l, l' \in L$ and $i, i' \in I$. 
\end{definition}

\begin{example}\label{ex999}
For any pair of Lie superalgebras $(L,I)$, the Lie multiplication induces an action of $L$ on $I$ via ${}^li=[l , i]$, and an action of $L$ on itself via ${}^ll'=[l , l']$ . 
\end{example}

 Let $(L,I)$ be a pair of Lie superalgebras with action on each others and let $T$ be a Lie superalgebra then we have the following defintion.

\begin{definition}
 A bilinear function $f:L\times I\longrightarrow T$ is called Lie pairing if the following relations are satidfied:
\begin{enumerate}
\item $f([l , l'], i)= f(l, {}^ {l'}{i}) -(-1)^{|l||l'|}f(l', {}^{l}{i})$,
\item  $f(l, [i,i'])=(-1)^{|i'|(|l|+|i|)}f({}^{i'}{l}, i)-(-1)^{|l||i|}f({}^{i}{l}, i')$,
\item $f({}^{i}{l} , {}^{l'}{i'})=-(-1)^{|l||i|}[f(l, i),f(l', i')]),$  
\end{enumerate} 
for every $l, l' \in L_{\overline{0}}\cup  L_{\overline{1}}$ and $i , i' \in I_{\overline{0}}\cup  I_{\overline{1}}$.
\end{definition}

%Given a pair of Lie superalgebras $(L,I)$  with action of $L$ on $I$, we define the semidirect product $I \rtimes L$ with underlying supermodule $I \oplus L$ endowed with the bracket given by
% $[(i, l), (i', l')]=([i, i']+ {}^li'-(-1)^{|i||l'|}({}^{l'}i), [l, l'])$. 
 
\begin{definition} 
 A crossed module of Lie superalgebras is a homomorphism of Lie superalgebras $\partial : I\longrightarrow L$ with an action of $L$ on $I$ satisfying 
\begin{enumerate}
\item $\partial ({}^l{i})=[l,\partial(i)],$
\item ${}^{\partial(i)}{i'}=[i, i'],$ for all $l \in L$ and $i, i' \in I$.
\end{enumerate} 
\end{definition}
For any pair of Lie superalgebras $(L,I)$, the canonical homomorphisms $\partial' : I\longrightarrow L$ and $\partial : L\longrightarrow L$ with the actions defined in Example \ref{ex999} are crossed modules and we called it the canonical crossed modules. 
Let $X_{L , I}$ be the $\mathbb{Z}_{2}$-graded set of all symbols $l\otimes i$, where $l \in L_{\overline{0}}\cup  L_{\overline{1}}$, $i \in I_{\overline{0}}\cup  I_{\overline{1}}$ and the $\mathbb{Z}_{2}$-gradation is given by $|l\otimes i|=|l|+|i|$.

\begin{definition}
The non-abelian tensor product of $L$ and $I$, denoted by $L \otimes I$, as the Lie superalgebra generated by $X_{L , I}$  and subject to the relations:
\begin{enumerate} 
 	\item $\lambda (l \otimes i)=\lambda l \otimes i= l \otimes \lambda i$,
 	\item $(l + l')\otimes i= l\otimes i +l'\otimes i$,  where $l , l'$ have the same degree,\\
 	$l \otimes (i + i')=l \otimes i + l \otimes i'$,  where $i , i'$ have the same degree,
 	
 	\item $[l , l']\otimes i= (l\otimes {}^ {l'}{i}) -(-1)^{|l||l'|}(l'\otimes {}^{l}{i})$,\\
 	$l\otimes [i,i']=(-1)^{|i'|(|l|+|i|)}({}^{i'}{l}\otimes i)-(-1)^{|l||i|}({}^{i}{l}\otimes i')$,
 	
 	\item $[l\otimes i,l'\otimes i']=-(-1)^{|l||i|}({}^{i}{l} \otimes {}^{l'}{i'}),$      
 \end{enumerate}
 for every $\lambda \in \mathbb{F}, l, l' \in L_{\overline{0}}\cup  L_{\overline{1}}$ and $i , i' \in I_{\overline{0}}\cup  I_{\overline{1}}$.
\end{definition}

It easy to check that the mapping \\
 \[L\times I\longrightarrow L\otimes I,\, (l,i)\longrightarrow l\otimes i\]
 is a lie pairing. The tensor product $L \otimes I$ has $\mathbb{Z}_{2}$-grading given by $(L \otimes I)_{\alpha}=\oplus_{\beta+\gamma=\alpha}  (L_{\beta}+I_{\gamma})$ for $\alpha, \beta, \gamma \in \Z_2$. If $L=L_{\overline{0}}$ and $I=I_{\overline{0}}$ then $L \otimes I$ is the non-abelian tensor product of Lie algebras introduced and studied \cite{Ellis1991}.

\begin{proposition}\label{prop636}\cite[Proposition 3.5]{GKL2015}
If the Lie superalgebras $M$ and $N$ act trivially on each other, then $M\otimes N$ is an abelian Lie superalgebra and there is an isomorphism of supermodules 
  	\[M\otimes N \cong M^{ab}\otimes _{\mathbb{K}}N^{ab},\]
where $M^{ab}=M/[M, M]$ and $N^{ab}=N/[N, N]$. 
\end{proposition}

\begin{proposition}\label{prop4}\cite[Proposition 3.8]{GKL2015}
Given a short exact sequence of Lie superalgebras
 $$(0, 0) \longrightarrow (K,L) \overset{(i, j)} \longrightarrow (M, N) \overset{(\phi, \psi)} \longrightarrow (P, Q) \longrightarrow (0, 0)$$ there is an exact sequence of Lie superalgebras 
 $$(K \otimes M) \rtimes (M \otimes L) \overset{\alpha} \longrightarrow M \otimes N \overset{\phi \otimes \psi} \longrightarrow P \otimes Q \longrightarrow 0.$$
 \end{proposition}

Specifically given a Lie superalgebra $M$ and a graded ideal $K$ of $M$ there is an exact sequence 
 \begin{equation}\label{eq1}
 	(K \otimes M) \rtimes (M \otimes K) \longrightarrow M \otimes M \longrightarrow (M/K) \otimes (M/K) \longrightarrow 0.
 \end{equation}

 \begin{definition}
Let $f$ be a lie pairing of a pair of Lie superalgebra $(L,I)$, with the canonical crossed modules, then it is called an exterior Lie pairing if the following relations are satisfied:
\begin{enumerate}
	\item $f(i, i')= -  (-1)^{|i||i'|}f(i', i)$,
	\item $f(i_{\overline{0}} ,i_{\overline{0}})=0$,
\end{enumerate} 
with $i,i' \in I_{\overline{0}} \cup I_{\overline{1}}, i_{\overline{0}} \in I_{\overline{0}}$.
 \end{definition}
Now let us consider a special case of \cite[Lemma 6.1]{GKL2015} when $\partial'$ and $\partial$ are the canonical crossed modules and is $I$ an ideal of $L$, then we have the flowing result.

\begin{proposition}
Let $L \square I$ be the submodule of $L \otimes I$ generated by elements
\begin{enumerate}
	\item $i\otimes i'+   (-1)^{|i||i'|}(i'\otimes i)$
	\item $i_{\overline{0}} \otimes i_{\overline{0}}$
\end{enumerate}
with $i,i' \in I_{\overline{0}} \cup I_{\overline{1}}, i_{\overline{0}} \in I_{\overline{0}}$. Then $L \square I$ is a central graded ideal of $L \otimes I$.
\end{proposition}

\begin{definition}
The exterior product of $L$ and $I$ is denoted as $L \wedge I$ and is defined as the quotient Lie superalgebra \[L\wedge I= \frac{L\otimes I}{ L\square I}.\] 
We denote the elements of $L\wedge I$ by  $l\wedge i$.
\end{definition}

Note that the exterior product of the pair $(L,I)$ satisfied all relations in the detention of the tensor product in addition :
\begin{enumerate}
 \item $i \wedge i'+   (-1)^{|i||i'|}(i'\wedge i)=0$
 \item $i_{\overline{0}} \wedge i_{\overline{0}}=0$
 \end{enumerate}
with $i,i' \in I_{\overline{0}} \cup I_{\overline{1}}, i_{\overline{0}} \in I_{\overline{0}}$.\\

An exterior Lie pairing $f:L\times I\longrightarrow T$ is said to be universal if for any other exterior Lie pairing $f' :L \times I\longrightarrow T'$ there is a unique Lie homomorphism $\phi: T \longrightarrow T'$ such that $\phi\circ f=f'$. The following proposition is analogue to Lie algebra case see \cite[proposition 12]{Ellis1987}.

\begin{proposition} \label{pro2.6}
The mapping 
	\[L\times I\longrightarrow L\wedge I,\, (l,i)\longrightarrow l\wedge i,\]
is a universal exterior Lie pairing.

\end{proposition}

\begin{lemma}\label{lem3.33}
Let $(L, I)$ be a pair of Lie superalgebras and $J$ is a graded ideal of $L$ with $J\subseteq I$. Then the following sequence is exact 
$$ J \wedge L \longrightarrow I \wedge L \longrightarrow L/J \wedge I/J \longrightarrow 0.$$
\end{lemma}
\begin{proof}
Let us consider the following canonical exact sequences 
\[ J  \longrightarrow  L \longrightarrow L/J\longrightarrow 0~ {\rm{and}}~
	  J  \longrightarrow I \longrightarrow  I/J \longrightarrow 0.\]
The result follows by taking $K=L=J,~ M=L, N=I,~ P=L/J,~ Q= I/J$ in the Proposition  \ref{prop4}. 
\end{proof}

\begin{lemma}\label{lem6.33}
Let $J$ and $K$ be two graded ideals of Lie superalgebra $L$ such that $K\cong J$. Then there is a homomorphism from $L\wedge K$ onto $L\wedge J$.
\end{lemma}
\begin{proof}
Suppose that $\alpha: J\longrightarrow K$ is an isomorphism. Define \\ 
$$ \beta:L\times J\longrightarrow L\wedge K ~{\rm{given~by}}$$
$(l,j)\longrightarrow l \wedge \alpha(j)$. It can be checked that $\beta$ is well defined exterior Lie pairing. By Proposition \ref{pro2.6} we have,
         	$$\beta': L\wedge J \longrightarrow  L\wedge K  ~{\rm{~given~by~}} ~l\wedge j \longmapsto l\wedge \alpha(j) $$ is a homomorphism, moreover it is onto.
\end{proof}

Let us consider the commutator map $[.,.]:L\wedge I\longrightarrow [L,I]~{\rm{~given~by~}} ~  l\wedge i\longrightarrow [l,i]$. Now we see how we can construct a representation for $L\wedge I$ from the representation of $L$.

\begin{lemma}\label{lem3.9}
Let $(L, I)$ be a pair of Lie superalgebras and let $0 \longrightarrow R \longrightarrow F \longrightarrow L$ be free presentation of $L$. If $I \cong S/R$ for some graded ideal $S$ of $F$, then $L \wedge I\cong [F, S]/[F,R]$.
\end{lemma}
\begin{proof}
Let us define the map \[ \varphi : L \wedge I\longrightarrow [F, S]/[F,R],\]
which can be easily seen to be epimorphism. Now from \cite[proposition 6.3]{GKL2015}  we have $\alpha: [F,F] \longrightarrow F \wedge F $ is an isomorphism. Then the restriction of $\alpha$ to $[F,S]$ is the homomorphism $\alpha \mid_{[F,S]}:[F,S]\longrightarrow F\wedge S$. Consider the epimorphism $F\wedge S\longrightarrow L\wedge I$, then we obtain the homomorphism
	\[   \bar{\pi}: [F,S]\longrightarrow L\wedge I\]
	\[[f,s]\longrightarrow f+R \wedge s+R\]
such that $[F,R] \subseteq \ker \bar{\pi}$. Hence $\phi : [F, S]/[F,R] \longrightarrow L \wedge I$ is a homomorphism such that $\varphi \phi = 1$ and $ \phi \varphi = 1$ and the result follows.
\end{proof}

\begin{corollary}\label{corr0303}
$L \wedge I\cong [L, \frac{S}{[F,R]} ]$
\end{corollary}
\begin{proof}
Since $[L,\frac{S}{[F,R]}]=[\bar{F},\bar{S}]=[F/[F,R],S/[F.R]]=[F,S]/[F,R]$, the result follows from Lemma \ref{lem3.9}.
\end{proof}

 \section{Multiplier of pair of Lie superalgebras}
 
In this section we will discuss on Schur multiplier of pair of Lie superalgebras. Consider the exact sequence of Lie superalgebras (see \cite{hesam})
\begin{align*}
&\mathcal{H}_3(L) \longrightarrow \mathcal{H}_3(L/I) \longrightarrow \mathcal{M}(L, I) \longrightarrow \mathcal{M}(L) \longrightarrow \mathcal{M}(L/I)\\ 
& \longrightarrow I/[L,I] \longrightarrow L/L^2 \longrightarrow L/[L^2+I] \longrightarrow 0.
\end{align*}
Then the schur multiplier of pair of Lie superalgebras $(L,I)$  is the abelian Lie superalgebra $\mathcal{M}( L,I)$ which appear in the above sequence. This definition is analogous to the definition of the Schur multiplier of a pair of groups and pair of Lie algebras \cite{Ellis1996, org}. Free presentation of a Lie superalgebra $L$ is the extension $0\longrightarrow R \longrightarrow F\longrightarrow L$, where $F$ is a free Lie superalgebra. If $I$ is a graded ideal of $L$ such that $I\cong S/R$ for some graded ideal $S$ of $F$, then $$\mathcal{M}( L,I) \cong R\cap [F,S] / [F,R].$$ 
From the above exact sequence, we have the following result.

\begin{example}\label{e1}
Consider the Heisenberg Lie superalgebra with even center $H(1,0)= \cong <x_1, x_2, z \mid [x_1 , x_2]=z>$.  We intend to give a free presentation of $H(1,0)$ and thereby compute the multiplier of $H(1,0)$. Consider a $\mathbb{Z}_{2}$-graded set $X=X_{\bar{0}}$, where $X_{\bar{0}}=\{x_1,x_2\}$. Let $F$ be the free Lie superalgebra generated by $X$.
As $H^3(1,0)=[H^2(1,0),H(1,0)]=0$, we have $F^3 \subseteq R$. Thus take $R=F^3$. Then $0 \longrightarrow R \longrightarrow F \longrightarrow H(1,0) \longrightarrow 0$  is the free presentation of $H(1,0)$. Now \[ [F,R]=F^4.\] 
Take $I=H^2(1,0) $be the central ideal of $H(1,0)$. Consider $S=F^2$,
then $S\supseteq R$ is a graded ideal of $F$ such that $I=S/R$. Hence
$0 \longrightarrow S \longrightarrow F \longrightarrow H(1,0)/I \longrightarrow 0$ is a free presentation of $H(1,0)/I$. So
$$ [S,F]=[F^2,F]=F^3. $$ 
Thus,
\[ \mathcal{M}(H(1,0),I)= ([F,S]\cap R)/ [F,R] = F^3 / F^4. \]
Now, by Theorem \ref{c12} dimension of $\dim \mathcal{M}(H(1,0),I)=2$.

\end{example}

\begin{example}\label{e2}
Now consider $H(0,1) \cong < z;y \mid [y,y]=z>$ and a $\mathbb{Z}_{2}$-graded set $X=X_{\bar{1}}$, where $X_{\bar{1}}=\{y\}$. Let $F$ be the free Lie superalgebra generated by $X$. Then $R=[y,z]+F^3$. Then $0 \longrightarrow R \longrightarrow F \longrightarrow H(0,1) \longrightarrow 0$  is the free presentation of $H(0,1)$. Now \[ [F,R]= <[[y,z],y],[[y,z],z]> + F^4.\] 

Take $I=H^2(0,1) $be the central ideal of $H_{1}$. Consider $S=F^2$,
then $S\supseteq R$ is a graded ideal of $F$ such that $I=S/R$. Hence
$0 \longrightarrow S \longrightarrow F \longrightarrow H(0,1)/I \longrightarrow 0$ is a free presentation of $H(0,1)/I$. So
$$ [S,F]=[F^2,F]=F^3. $$ 
Thus,
\[ \mathcal{M}(H(0,1),I)= ([F,S]\cap R)/ [F,R] = F^3 /[F,R]. \]
Using graded Jacobi identity, $[[y,y],z]+[F,R]$ is a basis for $\mathcal{M}(H(0,1),I)$, therefore  $\dim \mathcal{M}(H(0,1),I)=1$.

\end{example}

\begin{example}\label{e3}
Consider the Heisenberg Lie superalgebra with odd center $H_{1} \cong <x; y , z \mid [x , y]=z>$. Consider a $\mathbb{Z}_{2}$-graded set $X=X_{\bar{0}} \oplus X_{\bar{1}}$, where $X_{\bar{0}}=\{x\}$ and $X_{\bar{1}}=\{y\}$. Take $R=<[y,y]>+F^3$. Then we have a free presentation of $H_{1}$, $0 \longrightarrow R \longrightarrow F \longrightarrow H_{1} \longrightarrow 0$. Now \[ [F,R]=\big<[x,[y,y]] \big> + F^4,~~~F^3\cap R =F^3.\]
Consider $I=H^2_{1}=<[x, y]>$. Take $S=F^2$, then $S\supseteq R$ is a graded ideal of $F$ such that $I=S/R$. So
$$ [S,F]=[F^2,F]=F^3. $$ 
Thus,
\[ \mathcal{M}(H_{1},I)= ([F,S]\cap R)/ [F,R] = F^3 /[F,R]=
\big<[x,[x, y]]+[F,R],[y,[x, y]]+[F,R] \big>.\]
Using graded Jacobi identity, $-[y, [x, y]]+[x, [y, y]]+[y,[y, x]]=0$. As $[x, [y, y]] \in [F, R]$, implies $[y, [x, y]]=0$. Therefore $\{[x,[x, y]]+[F,R] \}$ is a basis of $\mathcal{M}(H_{1},I)$, hence  $\dim \mathcal{M}(H_{1},I)=1$.
\end{example}

\begin{lemma}\label{lem2.1}
Let $I$ and $I'$ be two ideals of the Lie superalgebra $L$ with $L=I\oplus I'$. Then
 \[\mathcal{M}(L)\cong \mathcal{M}(L, I) \oplus\mathcal{M}(I').\]
\end{lemma}

\begin{lemma}\label{lemma4}
Let $(L, I)$ be a pair of Lie superalgebras. Then 
 	\[0\longrightarrow  ~\mathcal{M}(L,I) \longrightarrow L\wedge I \longrightarrow [L,I]\longrightarrow 0 \]
is a central extension.
\end{lemma}
\begin{proof}
Let $F$ be a free Lie superalgebra on a graded set X. Let 
\begin{equation}\label{eq3.2}
0\longrightarrow R \longrightarrow F\longrightarrow L, 
\end{equation}
be a free presentation of Lie superalgebra $L$ with $I\cong S/R$ for some graded ideal $S$ of $F$.
Consider the map $\varphi:[F,S]/[F,R]\longrightarrow ([F,S]+R)/R$ defined as $x+[F,R]\longmapsto x+R$. Clearly $\varphi$ is onto and $\ker ~\varphi=[F,S]\cap R/[F,R]$. Now Let $x\in[F,S]\cap R $, then for all $y \in [F,S]$, we have  $[x, y] \in [F,R]$. So $[x+[F,R],y+[F,R]]=[x,y]+[F,R]=[F,R]$ implies that $[x+[F,R]]\in Z([F,S]/[F,R])$. Therefore
\[0\longrightarrow   ([F,S]\cap R)/[F,R] \longrightarrow [F,S]/[F,R] \longrightarrow ([F,S]+R)/R\longrightarrow 0 \]
is a central extension. From the free presentation of $L$, we have  $0\longrightarrow R \longrightarrow S\longrightarrow I$. From Proposition \ref{prop4}, we have the exact sequence 
\[0\longrightarrow F\wedge R \longrightarrow F\wedge S\longrightarrow L\wedge I.\]
Now from Lemma \ref{lem3.9}, we have $L\wedge I \cong [F,S]/[F,R]$. We know $\mathcal{M}(L,I)= ([F,S]\cap R)/[F,R]$ and since $L\cong F/R$, \, $I\cong S/R$, thus we have $[L,I]\cong([F,S]+R)/R$. Hence $0\longrightarrow  ~\mathcal{M}(L,I) \longrightarrow L\wedge I \longrightarrow [L,I]\longrightarrow 0 $ is a central extension.  
\end{proof}

The exact sequence in Lemma \ref{lemma4} can be used to characterized the Schur multiplier of pair of Lie superalgebras.

\begin{corollary}
Let $I$ be a graded ideal of an abelian Lie superalgebra $A(m \mid n)$. Then 
	\[\mathcal{M}(A(m \mid n),I)\cong  A(m \mid n)\wedge I\]
\end{corollary}
\begin{proof}
The result follows from Lemma \ref{lemma4}, since $[\mathcal{M}(A(m \mid n),I]=0$.
\end{proof}

\begin{corollary}\label{corr999}
Let $I$ be a $(k\mid h)$-dimensional ideal of $(A(m \mid n)$. Then 
	\[\dim \mathcal{M}(A(m \mid n),I)=(\frac{1}{2}[k(2m-k-1)+h(2n-h+1)] \mid mn-(m-k)(n-h)).\] 
\end{corollary}

\begin{proof}
Let $I$ be a $(k\mid h)$-dimensional ideal of $(A(m \mid n)$. Then there exist an ideal $Q$ such that $A(m \mid n)=I\oplus Q$. Now, from Lemma \ref{lem2.1} we have 
	\[\mathcal{M}(A(m \mid n))\cong \mathcal{M}(A(m \mid n),I) \oplus\mathcal{M}(Q)\]
where $\dim Q=(m-k \mid n-h)$. Then Theorem \ref{th3.3} implies that
\begin{align*}
\dim\mathcal{M}(A(m \mid n),I)=&(\frac{1}{2}(m^2+n^2+n-m) \mid mn)-(\frac{1}{2}[(m-k)^2+(n-h)^2+n-h-(m-k)]\mid (m-k)(n-h)] \\
=&(\frac{1}{2}[k(2m-k-1)+h(2n-h+1)] \mid mn-(m-k)(n-h))
\end{align*}
\end{proof}

In the previous result, if we take $h=n=0$, then we have:
\begin{corollary}\cite[ Corollary 2.1.13]{org}
Let $I$ be a $k$-dimensional ideal of $A(m)$. Then 
	\[\dim \mathcal{M}(A(m),I)=\frac{1}{2} k(2m-k-1) .\] 
\end{corollary}

\begin{lemma}\label{lemm9875}
Let $I$ be a central graded ideal of $L$. Then $\mathcal{M}(L,I) \cong L\wedge I\cong (L/L^2\otimes I)/\left\langle (x+L^2)\otimes x \mid x \in I \right\rangle $.
Moreover, if $I\subseteq L^2$, then $\mathcal{M}(L,I) \cong  (L/L^2\otimes I)$
\end{lemma}
\begin{proof}
As $[L, I] = 0$, from Lemma \ref{lemma4} $\mathcal{M}(L,I) = L\wedge I$. Now from Proposition \ref{prop636}, we conclude the proof.
\end{proof}
%For any two Lie superalgebras $H$ and $K$ the Lie superalgebra direct sum $H \oplus K$ is a Lie superalgebra with natural grading $(H \oplus K)_{\alpha}=H_{\alpha}\oplus K_{\alpha}$ where $\alpha \in \mathbb{Z}_2$. If $\mathcal{M}(H)$ and $\mathcal{M}(K)$ are known, then $\mathcal{M}(H \oplus K)$ is given by the following result.

%\begin{theorem}\label{th3.7}\cite[See Theorem 3.9]{Nayak2018}
%\[\mathcal{M}(H\oplus K)\cong \mathcal{M}(H)\oplus \mathcal{M}(K)\oplus (H/H'\otimes K/K').\]
%\end{theorem}

\section{ Capability of a pair of Lie Superalgebras}

The concept of epicenter $Z^*(L)$ of a Lie algebra $L$ was defined in \cite{Alamian2008} and the analogous notion of the epicenter $Z^*(L,I)$ for a pair of Lie algebra $(L,I)$ was defined in \cite{org} and it has been proved that $(L,I)$  is capable if and only if $Z^*(L,I)=0$. In this section, we define the epicenter $Z^*(L,I)$ for the pair of Lie superalgebras $(L, I)$ and characterized $Z^*(L,I)$.

\begin{definition}
Let $(L,I)$ be a pair of Lie superalgebras. A Lie homomorphism $\beta:N\longrightarrow L$ together with an action of  $L$ on $N$ denoted by ${}^{l}{n}$ is a relative central extension of the pair $(L,I)$ if the following conditions holds:
\begin{enumerate}
	\item $\beta(N)=I$.
	\item $\beta({}^{l}{n})=[l,\beta(n)]$.
	\item ${}^{\beta (n_l)}{n}=[n_1,n]$.
	\item $\ker \beta\subseteq Z[L,N]$, where  $Z[L,N]=\{ n \in N \mid {}^{l}{n} =0, ~\forall ~l \in L\} $
\end{enumerate}
for all $n,n_1 \in N$ and $l\in L$.	
\end{definition}

The relative central extension $\beta:N\longrightarrow L$ is universal if there is any other central extension $\alpha:M\longrightarrow L$, then there exists a unique homomorphism $\phi:N\longrightarrow M$ such that $\alpha \phi =\beta$ and $\phi({}^{l}{n})={}^{l}{\phi(n)}$.

\begin{example}\label{ex4.7}
Let us consider the same assumption in the Lemma \ref{lem3.9} and the action of $L$ on
	$S/[F, R]$ by ${}^{l}{s+[F,R]}=[f,s]+[R,F]$ where $\pi(f)=l$. Define the Lie superalgebra homomorphism
	\[\varphi:\frac{S}{[F,R]}\longrightarrow L ~~{\rm{by}}~~ s+[F,R]\longmapsto\pi(s)\]
where $\pi$ is the canonical epimorphism from $F$ to $L$. One can easily check that $\varphi$ satisfy all above conditions with $\varphi(\frac{S}{[F,R]})=\pi(S)=S/R\cong I$, hence $\varphi$ is a relative central extension of $(L,I)$.
\end{example}

\begin{definition}
The pair $(L,I)$ is capable if there exists a relative central extension $\beta:N\longrightarrow L$  with $\ker \beta = Z[L,N]$.
\end{definition}

\begin{definition}
Let $I$ be a graded ideal of Lie superalgebra $L$. The epicenter of the pair $(L,I)$ is denoted by $Z^*(L,I)$ and is defined as follows:
	 \[Z^*(L,I)=\{\cap \beta(Z(L,H)) \mid \beta: H\longrightarrow L~ {\rm{is ~a~ relative~ central~ extension~of~the ~pair~}} (L,I)\}.\]

\end{definition}
Note that, if $I=L$, then  $Z^*(L,I)= Z^*(L)$ which has been studied in \cite{Padhandetec}. Now we characterize capable Lie superalgebras. The proof of the following results are analogous to pairs of groups and pair of Lie algebras, so we omit the proof.

\begin{theorem}\label{lemma1}
Let $(L,I)$ be a pair of Lie superalgebras and $\{J_{k}\}_{k \in K}$ be a family of graded ideals of a Lie superalgebra $L$. If for each $k \in K$, the pair of quotient Lie superalgebras $(L/J_{k},I/J_{k})$ is capable, then so is $(L/ \cap_{k \in K} J_{k},I/ \cap_{k \in K} J_{k})$.
\end{theorem}

\begin{lemma} \label{lem44}
Let $(L,I)$ be a pair of Lie superalgebras. Then $(L/Z^{*}(L,I),I/ Z^{*}(L,I))$ is capable.
\end{lemma}

\begin{corollary}\label{lemma0}
A pair of Lie superalgebras $(L,I)$ is capable if and only if $Z^{*}(L,I)=\{0\}$. 
\end{corollary}

\begin{theorem}
Let $(L,I)$ be a pair of Lie superalgebras. Then $Z^{*}(L,I)$ is the intersection of all graded ideals $J$ of $L$ such that $(L/J,I/ J)$ is capable.
\end{theorem}
\begin{proof}
Let $J$ be a graded ideal of $L$ such that $(L/J,I/ J)$ is capable. Then we have a relative central extension $\beta:H\longrightarrow L/J$, with $\ker \beta=Z(L/J,H)$. Now, let us define \[E=\{(l,h)\in L\oplus H \mid \beta(h)=l+J\}\]
with an action of $L$ on $E$ given by $(l,h)^{l'}=(l^{l'},h^{l'+J})$ for all $l,l' \in L$ and $h \in H$. Let us consider $\varphi: E\longrightarrow L~{\rm{given~by}}~ (l,h)\longmapsto l$. Clearly, $\varphi$ is relative central extension of $(L,I)$, hence $Z^*(L,I)\subseteq \varphi(Z(L,E))$. On the other hand, if $(l,h) \in Z(L,E)$, then $\beta(h)=l+J=J$ as $h \in Z(L/K, H)=\ker \beta$. Thus $\varphi(l,h)=l \in J$. Therefore, $Z^*(L,I)\subseteq J$ and by Lemma \ref{lem44} the result follows. 
\end{proof}

\begin{theorem}
Let $(L_k,J_k$) be a family of pairs of Lie superalgebras. Then 
	\[Z^*(\oplus_{k\in K}L_k,\oplus_{k\in K}J_k)\subseteq \oplus_{k\in K}Z^*(L_k,J_k).\]
\end{theorem}
\begin{proof}
Let $\beta_k: H_k\longrightarrow L_k$ be an arbitrary relative central extension of $(L_k,J_k)$ for all $k \in K$. Now let $L=\oplus_{k \in K}L_K, \, J=\oplus_{k \in K}J_k$ and $H=\oplus_{k \in K}H_k$. If we define $\phi:H\longrightarrow L$ by $\{h_k\}_{k\in K}\longmapsto \{\beta_k(h_k)\}_{k\in K}$, then $\phi$ is a relative central extension of $(L,J)$ and $\phi(Z(L,H))=\oplus_{k \in K}\beta_k(Z(L_k,H_k))$. Therefore, $Z^*(L,J)\subseteq \oplus_{k \in K}\beta_k(Z(L_k,H_k))$. Thus $Z^*(L,J)\subseteq \oplus_{k \in K}Z^*(L_k,J_k)$, since $\beta_k$ is arbitrary.
\end{proof}

\begin{corollary}
Let $(L_k,J_k)$ be a capable pair of Lie superalgebras for all $k\in K$. Then the pair $(\oplus_{k\in K}L_k,\oplus_{k\in K}J_k)$ is capable.
\end{corollary}

\begin{definition}
Let $(L,I)$ be a pair of Lie superalgebras, the exterior $L$-center of $I$ is 
\[Z^{\wedge}_{L}(I)=\{i\in I \mid l\wedge i=0~{\rm{for~all}}~ l \in L\}\]
\end{definition}

It can be observed that if $L=I$, then $Z^{\wedge}_{L}(I)$ is the exterior $L$-center of Lie superalgebra $L$ which has been studied in \cite{Padhandetec}

\begin{theorem}\label{thm4.17}
Let $K$ be a graded ideal of $L$ such that $I\subseteq K$, then $K\subseteq Z^{\wedge}_{L}(I)$ if and only if the natural map $L\wedge I\longrightarrow L/K \wedge I/K$ is a monomorphism or equivalently the map $\mathcal{M}(L,I)\longrightarrow \mathcal{M}(L/K , I/K )$ is a monomorphism.
\end{theorem}
\begin{proof}
Considering the following commutative diagram
\begin{center}
	\begin{tikzpicture}[>=latex]
		\node (x) at (0,0) {\(\mathcal{M}(L,I) \)};
		\node (z) at (0,-3) {\(\mathcal{M}(L/K , I/K )\)};
		\node (y) at (3,0) {\(L\wedge I\)};
		\node (w) at (3,-3) {\(L/K \wedge I/K\)};
		\draw[->] (x) -- (y) node[midway,above] {$\iota$};
		\draw[->] (x) -- (z) node[midway,left] {$\beta$};
		\draw[->] (z) -- (w) node[midway,below] {$\iota$};
		\draw[->] (y) -- (w) node[midway,right] {$\alpha$};
	\end{tikzpicture}\\
 
\end{center} 
From the above commutative diagram one can easily observe that $\beta$ is monomorphism if and only if $\alpha$ is. Now the proof completes, by using  Lemma \ref{lem3.33}   .
\end{proof}

The following result gives us what the connection between the exterior $L$-center of pair $(L,I)$ and the exterior $L$-center of Lie superalgebra $L$.

\begin{proposition}\label{prop2.3.5}
Let $(L, I)$ be a pair of Lie superalgebras. Then $Z^{\wedge}_{L}(I)\subseteq Z^{\wedge}(L)$.
\end{proposition}
\begin{proof}
Let $i\in Z^{\wedge}_{L}(I)$, then $i\in I\subseteq L$ with $i\wedge l=0$ for all $l \in L$, then $i \in Z^{\wedge}(L)$ and the result follows.
\end{proof}

\begin{proposition}\label{prop6}
Let $(L,I)$ be a pair of Lie superalgebras and $K$ be an ideal of $L$ containing $I$. Let  $0\longrightarrow R \longrightarrow F\overset{\pi}{\longrightarrow} L\longrightarrow 0$ be a free presentation. If $S$ and $T$ are two ideals in $F$ with $S/R\cong I$,  $T/R\cong K$, then the following sequence is exact:
$0\longrightarrow \frac{R\cap [T,F]}{[R,F]} \longrightarrow \mathcal{M}(L,I)  \overset{\sigma} \longrightarrow \mathcal{M}(L/K,I/K)\longrightarrow \frac{K\cap (L,I)}{[K,L]}\longrightarrow 0$. Moreover, if $K$ is a central ideal, then $K\subseteq Z^*(L,I)$ if and only if $\frac{\mathcal{M}(L/K,I/L)} {\mathcal{M}(L,I)}\cong K\cap [L,I]$.
\end{proposition}
\begin{proof}
With the given assumption in the theorem we have $\mathcal{M}(L,I)= \frac{R \cap [S,F]}{[R,F]}$. Since $F/R \cong L$, so $L/K \cong F/T$ and $0 \longrightarrow T/R \longrightarrow F/R \longrightarrow F/T \longrightarrow 0$ is free presentation of $L/K$, then $\mathcal{M}(L/K,I/K)=\frac{T \cap [S,F]}{[T,F]}$. Now $\frac{K \cap [I,L]}{[K,L]}\cong \frac{T/R \cap [S/R,F/R]}{[T/R,F/R]} \cong  \frac{T \cap [S,F]+R}{[T,F]+R}$. The map $ \frac{R \cap [T,F]}{[R,F]}\xrightarrow{\alpha} \frac{R \cap [S,F]}{[R,F]}$ is an inclusion map and ${\rm{im}} (\alpha)=\frac{R \cap [T,F]}{[R,F]}$. If we define $\frac{R \cap [S,F]}{[R,F]}\xrightarrow{\beta} \frac{T \cap [S,F]}{[T,F]}$ by $x+[R,F]\longmapsto x+[T,F]$, then $\beta$ is well defined and a Lie superalgebra homomorphism with $\ker \beta=\frac{R \cap [T,F]}{[R,F]}$ and ${\rm{im}} (\beta)=\frac{R\cap [S,F]}{[T,F]}$. Finally, consider the map $\frac{T \cap [S,F]}{[T,F]}\xrightarrow{\phi}\frac{T\cap [S,F]+R}{[T,F]+R}$ by $y+[T,F]\longmapsto y+([T,F]+R)$. Here, $\ker (\phi)=\frac{R\cap [S,F]}{[T,F]}$ and ${\rm{im}} (\phi)=\frac{T\cap [S,F]+R}{[T,F]+R}\cong \frac{K \cap [I,L]}{[K,L]}$. We conclude that the given sequence is an exact sequence.\\

Let $K$ be a central ideal of $L$, then $[L,K]=0 \Leftrightarrow [T/R,F/R]=R \Leftrightarrow [T,F]\subseteq R$ $\Leftrightarrow$ $\frac{R\cap [T,F]}{[R,F]}=\frac{[T,F]}{[R,F]}$. On the other hand,  $K\subseteq Z^*(L,I)$ if and only if $\mathcal{M}(L,I)\longrightarrow \mathcal{M}(L/K,I/L)$ is monomorphism, and from the exact sequence in the first part of this theorem this is equivalent to $[T,F]=[R,F]$. Then we have $\ker\phi=\frac{R\cap [S,F]}{[R,F]}=\mathcal{M}(L,I)$ and this is true if and only if $\frac{\mathcal{M}(L/K,I/K)}{\mathcal{M}(L,I)}\cong K\cap[L,I]$. 
\end{proof}

\begin{corollary}\label{coro852}
With same assumption in the above theorem we have $K\subseteq Z^*(L,I)$ if and only if 
$\dim \mathcal{M}(L,I)=\dim \mathcal{M}(L/K,I/K)-\dim (K\cap [L,I])$.
\end{corollary}

\begin{theorem}
Considering the Lie superalgebra homomorphism defined in the Example \ref{ex4.7}, we have $Z^{\wedge}_{L}(I)=\varphi(Z(L,S/[R,F]))$.
\end{theorem}
\begin{proof}
From Corollary \ref{corr0303} we have, $[\bar{s},l]=0$ if and only if $\pi(s)\wedge l=0$ for all $s \in S$ and $l \in L$. Thus, 
\begin{align*}
	\varphi(Z(L,S/[R,F]) & =\{\varphi(\bar{s}) \mid [\bar{s},l]=0 ~{\rm{for~all}}~ l \in L\}\\
	&=\{\pi(s) \mid \pi(s)\wedge l=0 ~{\rm{for~all}}~ l\in L\}\\
	&=Z^{\wedge}_{L}(I).
\end{align*}
\end{proof}

Now we will use the notation of universal central extensions of Lie superalgebras to proof an important theorem.

\begin{corollary}
Let $0\longrightarrow R \longrightarrow F\overset{\pi}{\longrightarrow} L\longrightarrow 0$ be a free presentation of the pair of Lie superalgebras $(L,I)$ . If $\varphi:\frac{S}{[F,R]}\longrightarrow L$ is a universal relative central extension, then 
\[Z^{\wedge}_{L}(I)=Z^*_{L}(I).\]
\end{corollary}
\begin{proof}
Let $\rho:N\longrightarrow L$ be an arbitrary relative central extension of $(L,I)$. As $\varphi$ is universal, there exists a homomorphism $\phi:\frac{S}{[F,R]}\longrightarrow N$ such that $\rho \phi =\varphi$ and $\phi({}^{l}{\bar{s}})={}^{l}{\phi(\bar{s})}$ for all $l \in L$ and $\bar{s}\in S/[R,F]$. Thus, if $\bar{s}\in Z(L,S/[R,F])$, then $\phi(\bar{s})\in Z(L,N)$. Then $\varphi(\bar{s})=\rho o\phi(\bar{s})\in \rho(Z(L,N))$, hence $\varphi(Z(L,S/[R,F]))\subseteq Z^*(L,I)$. Therefore, $\varphi(Z(L,S/[R,F]))= Z^*(L,I)$, since $\varphi$ is a relative central extension of $(L,I)$. Finally, from the above theorem the result follows. 
\end{proof}

\begin{corollary}\label{cor4.22}
The pair $(L,I)$ is capable if and only if $Z^{\wedge}_{L}(I)=0$
\end{corollary}

In \cite{Padhandetec} Padhan proved that for any finite dimensional Lie superalgebra $L$, we have $Z^*(L)=Z^\wedge (L)$ and $L$ is capable if and only if $Z^\wedge (L)=0$. Then from Proposition \ref{prop2.3.5} and Corollary \ref{cor4.22} it can be observed that if $L$ is capable, then the pair $(L,I)$ is also capable.

\begin{corollary}
The pair of Lie superalgebras $(L, I)$ is capable if and only if $\mathcal{M}(L,I)\xrightarrow{\delta} \mathcal{M}(L/\left\langle z \right\rangle , I//\left\langle z \right\rangle)$ has a non-trivial kernel for all non-zero elements $z \in Z(L)\cap I$.
\end{corollary}
\begin{proof}
Let $(L,I)$ be capable, then $Z^{\wedge}_{L}(I)=0$. If we assume that there is a non-zero $z\in Z(L)\cap I$ such that $\mathcal{M}(L,I)\xrightarrow{\delta} \mathcal{M}(L/\left\langle z \right\rangle   ,I//\left\langle z \right\rangle)$ is monomorphism, then $\left\langle z \right\rangle \subseteq Z^{\wedge}_{L}(I)=0$, by Theorem \ref{thm4.17}, this lead us to a contradiction. Conversely, if we assume that $(L,I)$ is not capable, then $Z^{\wedge}_{L}(I)\neq0$. So, there exists a non-zero $z \in Z^{\wedge}_{L}(I)$ such that $\mathcal{M}(L,I)\xrightarrow{\delta} \mathcal{M}(L/\left\langle z \right\rangle   ,I//\left\langle z \right\rangle)$ is a monomorphism, by Theorem \ref{thm4.17}, this give us a contradiction. Therefore, the result is true.
\end{proof}

The following lemma is analogous to the pair of Lie algebras case (see \cite{org}), so we state it without the proof.

\begin{lemma}\label{lemm99}
Let $(L, I)$ be a pair of finite dimensional Lie superalgebras with $[L, I]\neq I $.  Then the pairs $(L/[L,I],I/[L,I])$ and $(L,Z^{\wedge}_{L}(I))$ are capable. Moreover, $Z^{\wedge}_{L}(I)\subseteq [L,I]$.
\end{lemma}

\section{Capability and Schur Multiplier of a pair of abelian and Heisenberg Superalgebras}
In this section, we discuss on capability of pairs of Lie superalgebras whose derived subalgebra has dimension one and a non-abelian ideal and Schur multiplier of pairs of Heisenberg Lie superalgebras. First, we show that every finite dimensional non-abelian Lie superalgebra can be decomposed into the direct sum of a non-abelian Lie superalgebra and an abelian one.

\begin{definition}
The Lie superalgebra $L$ is a central product of $I$ and $J$, if $L$ is generated by $I$ and $J$, where $I$ and $J$ are graded ideals of $L$ such that $[I,J]=0$, $I\cap J\subseteq Z(L)$. We denote the central product of $I$ and $J$ by $L=I\dotplus J$.
\end{definition}

The following lemma shows that a Heisenberg Lie superalgebra is a central product of its ideals.

%\begin{lemma}\label{lemm987}
%Let $H(m \mid n)$ be a Heisenberg Lie superalgebra with even center. Then $H(m \mid n)=H(m_1\mid n_1)\dotplus H(m-m_1 \mid n-n_1)$ for some $1\leq m_1 \leq m$ and $1\leq n_1 \leq n$. Similarly, if $H_m$ is a Heisenberg Lie superalgebra with odd center, then $H_m=H_{m_1}\dotplus H_{m-m_1}$ for some $1\leq m_1 \leq m$.  
%\end{lemma}

\begin{lemma}\label{lemm987}
$H(m , n)$ is central product of its even Heisenberg ideals and $n+1$ dimensional Heisenberg ideal.
\end{lemma}
\begin{proof}
From Theorem \ref{th3.4}, consider $I_i =<x_i , x_{i+1},z>$ for $1\leq i \leq m$ and $J=<z; y_1, \ldots , y_n>$.
\end{proof}

\begin{lemma}\label{lemm988}
$H_m$ is central product of $H_1$.
\end{lemma}
\begin{proof}
From Theorem \ref{th3.6}, consider $I_i =<x_i ; y_i ,z> \cong H_1$ for $1\leq i \leq m$.
\end{proof}

Now we will show that there always exist two capable pairs for every pair of finite dimensional Lie superalgebras.

\begin{proposition}\label{prop1}
Let $L=I\dotplus J$ be a finite dimensional Lie superalgebra such that $I^2\cap J^2\neq 0$ and $K$ is an ideal of $I$ such that $[K,I]\subseteq I^2\cap J^2$. Then the pairs $(L, I),~ (L, J)$ and $(L, K)$ are non-capable. 
\end{proposition}
\begin{proof}
Let $l=i+j \in L$ be an arbitrary element, where $i\in I, ~j\in J$, and let $a \in I^2\cap J^2$ be a non zero element. Then we have $a=\sum_{r=1}^{n}\alpha_r[i_r,i'_r]=\sum_{s=1}^{m}\beta_s[j_s,j'_s]$ where $i_r,i'_r\in I$, $j_s,j'_s\in J$ and $\alpha_r,\beta_s$ are scalers. Now since $[I,J]=0$, we have
	 \[j\wedge a=j\wedge \sum_{r=1}^{n}\alpha_r[i_r,i'_r]=
	 \sum_{r=1}^{n}\alpha_r((-1)^{|i'_r|(|j|+|i_r|)}([i'_r,j]\wedge i_r)-(-1)^{|j||i_r|}([i_r,j]\wedge i'_r))=0\]
	 and 
	 \[i\wedge a=i\wedge \sum_{s=1}^{m}\beta_s[j_s,j'_s]=\sum_{s=1}^{m}\beta_s((-1)^{|j'_s|(|i|+|j_s|)}([j'_s,i]\wedge j_s)-(-1)^{|i||j_s|}([j_s,i]\wedge j'_s))=0,\]
then $l\wedge a =0$, for all $l \in L$. Hence $Z^{\wedge}_L(I)\neq 0$. Similarly, we have $Z^{\wedge}_L(J)\neq 0$ and $Z^{\wedge}_L(K)\neq 0$, now the result follows from Corollary \ref{cor4.22}.
\end{proof}

The result below tells us a  comprehensive classification of capable pairs of abelian Lie superalgebra.

\begin{theorem}\label{prop2}
Let $I$ be a $(k\mid h)$-dimensional ideal of $A(m \mid n)$ and $k+h \geq 1$. Then $(A(m \mid n), I)$ is capable if and only if $m=0$, $n=1$ or $m+n \geq 2$.  
\end{theorem}
\begin{proof}
Consider $m=1$ and $n=0$, then from Theorem \ref{th4.22}, $(A(1 \mid 0),A(1 \mid 0))$ is not capable pair. Now if $m=0$ and $n=1$ or $m+n\geq 2$. Then by Theorem \ref{th4.22} and Proposition \ref{prop2.3.5}, we have $Z^{\wedge}_{A(0 \mid 1)}(A(0 \mid 1))\subseteq Z^{\wedge}(A(0 \mid 1))=0$ or $Z^{\wedge}_{A(m \mid n)}(I)\subseteq Z^{\wedge}(A(m \mid n))=0$, thus the result can be concluded by using Corollary \ref{cor4.22}.  
\end{proof}

Now we shall give a complete classification of a capable pair $(L, I)$ for $L$ Heisenberg Lie superalgebras.

\begin{lemma}\label{lemm555}
Let $I$ be a graded ideal of $H(1 , 0)$ and $J$ is a graded ideal of $H_1$. Then $(H(1, 0), I)$ and $(H_1,J)$ are capable.
\end{lemma}
\begin{proof}
The result follows immediately from Theorems \ref{th4.33}, \ref{thm555} and Proposition \ref{prop2.3.5}.
\end{proof}

Let $I$ be an graded ideal of Heisenberg Lie superalgebras $L$. Then $I$ may be abelian or non-abelian. We can study separately for the  pair $(L, I)$ where $I$ is a non-abelian and abelian.

%In the following proposition, we study when the  the pair $(L, I)$ for $L$ Heisenberg Lie superalgebras  with even center and  $I$ is a non-abelian graded ideal  is not capable.s.

\begin{theorem}
For $m+n \geq 2$, the pair $(H(m, n), I)$ is not capable, when $I$ is a non-abelian graded ideal.
\end{theorem}
\begin{proof}
Let $\dim I =(k\mid h)$, then by Lemma \ref{lemm78} we have either $I$ is Heisenberg or $I \cong H(r,s)\oplus A(k-2r-1\mid h-s)$ for some $1\leq r \lneq m$ and $1\leq s \lneq n$. If $I$ is Heisenberg, say $I=H(m_1,n_1)$ for some $1\leq m_1 \leq m$ and $1\leq n_1 \leq n$, then from Lemma \ref{lemm987}, we can have $H(m , n)=H(m_1, n_1)\dotplus H(m-m_1 , n-n_1)$. Now, by Proposition \ref{prop1}, we have $(H(m, n), I)$ is not capable. If $I\cong H(r,s)\oplus A(k-2r-1\mid h-s) $, then similarly  from Lemma  \ref{lemm987}, we have  $H(m , n)=H(r, s)\dotplus H(m-r , n-s)$. Hence $A(k-2r-1\mid h-s)$ is a graded ideal of $H(m-r , n-s)$ and from Proposition \ref{prop1} $(H(m , n),A(k-2r-1\mid h-s))$ is not capable. Hence $Z^{\wedge}_{H(m , n)}(A)\neq 0$. Also, from the first case of the proof we have $Z^{\wedge}_{H(m , n)}(H(r,s))\neq 0$ then there are $0\neq a \in A(k-2r-1\mid h-s)$ and $0\neq b \in H(r,s)$ such that $ i\wedge x=0$ for all $x \in H(m , n)$ where $0\neq i=(b+a)$. Therefore $Z^{\wedge}_{H(m , n)}(I)\neq 0$ and the result follows from Corollary 
 \ref{cor4.22}.
\end{proof}

Similarly we can prove for the case of Heisenberg Lie superalgebras with odd center.

\begin{theorem}
For $m \geq 2$, the pair $(H_m, I)$ is not capable, when $I$ is a non-abelian graded ideal.
\end{theorem}

Now we should investigate the situation in which $I$ is an  abelian graded ideal.

\begin{proposition}
The pair $(H(m, n), I)$ is not capable, when $I$ is an abelian  graded ideal, $\dim I =(k\mid h)$ where $ h+k\geq 2  $ and $m+n\geq 2 $.
\end{proposition}
\begin{proof}
From Lemma \ref{lemm987}, we can get $H(m , n)=H(m_1, n_1)\dotplus H(m-m_1 , n-n_1)$. Putting $H(m_1, n_1)=J$  and $B^2=\left\langle b \right\rangle $ and let $I=\left\langle a, b \right\rangle$, for $a \in B$. So the pair $(H(m, n), A)$ is non-capable, by Proposition \ref{prop1}. Now let $I=\left\langle s_1,s_2,...,s_{m-1},z,r_1,r_2,...r_h \right\rangle$  and $B_1=\left\langle z,s_1 \right\rangle$. Then by Lemma \ref{lem6.33} there exists the homomorphism $H(m , n)\wedge A \longrightarrow H(m , n) \wedge B_1$. Hence, we have $Z^{\wedge}_{H(m , n)}(B_1)\neq 0$, then there exist $0\neq b_1 \in B_1 \subseteq I$ suth that  $b_1\wedge h^*=0$ for all $h^* \in H(m , n) $. Hence  $Z^{\wedge}_{H(m , n)}(I)\neq 0$ and  $(H(m , n),I)$ is not capable, by Corollary \ref{cor4.22}.
\end{proof}

Now for the case of Heisenberg Lie superalgebras with odd center we have the following result where its proof similar to the case of even center.

\begin{proposition}
The pair $(H_m, I)$ is not capable, when $I$ is an abelian  graded ideal, $\dim I =(k\mid k+1)$ where $ k\geq 1  $ and $m\geq 2 $.
\end{proposition}

\begin{lemma}
The pair $(L,L^2)$ is capable for all Heisenberg Lie superalgebra $L$.
\end{lemma}
\begin{proof}
Let us consider the following cases:
\begin{enumerate}
	\item $L=H(m , n)$ with $m=1,n=0,$
	\item $L=H(m , n)$ with  $m+n \geq 2$, or $L=H_m$ with $m\geq1,$
	\item $L=H_m$ with $m=1,$
\end{enumerate}

{\bf Case-1:}  The result follows ,by using Lemma \ref{lemm555}.

{\bf Case-2:} We have $Z^{\wedge}_{L}(L^2)\subseteq Z(L,L^2) \subseteq Z(L) =L^2$. Since $\dim L^2=1$,  then either $Z^{\wedge}_{L}(L^2)=0$ or $0\neq Z^{\wedge}_{L}(L^2)$. If  $Z^{\wedge}_{L}(L^2)=0$, then result follows from Corollary \ref{cor4.22}. If $(L,Z^{\wedge}_{L}(L^2)=L^2)$, then by Lemma \ref{lemm99} we have $(L,Z^{\wedge}_{L}(L^2)=L^2)$ is capable and this contradiction that $0\neq Z^{\wedge}_{L}(L^2)$.

{\bf Case-3:}  The result follows directly from Lemma \ref{lemm555}.
\end{proof}

The following theorem gave us the necessary and sufficient condition for the capability of Heisenberg Lie superalgebras with even center.

\begin{theorem}\label{the0909}
The pair $(H(m , n), I)$ is capable if and only if $m = 1$ and $n=0$ or $dim I = 1$.
\end{theorem}

Now we state the necessary and sufficient condition for the capability of Heisenberg Lie superalgebras with odd center.

\begin{theorem}
The pair $(H_m, I)$ is capable if and only if $m = 1$ and  or $dim I = 1$.
\end{theorem}

Now we are going to obtain the Schur multiplier of the pair $(H(m), I)$.

\begin{theorem}
Let $I$ be a $(k\mid h)$-dimensional graded ideal of $H(m , n)$ and $m+n \geq 1$. Then 
\begin{enumerate}
\item If $k+h = 1$, then $\dim \mathcal{M}(H(1 , 0), I) = 2$
\item If $k+h = 1$, then $\dim \mathcal{M}(H(0 , 1), I) = 1$
\item If $k+h = 1$, then $\dim \mathcal{M}(H(m , n), I) = 2m+n$
\item If $k+h \geq 2$, then 
\begin{align*}
	\dim \mathcal{M}(H(m , n), I) & =(\frac{1}{2}[(k-1)(4m-k)+h(2n-h+1)-1]\mid 2mn-(2m-k+1)(n-h)).
\end{align*}
\end{enumerate}
\end{theorem}

\begin{proof}
{\bf Case-1:} Observe that if $I$ is an ideal of $H(1,0)\cong <x_1,x_2,z>$, then $I=<x_1,[x_1,x_2]>$ or $I=<x_2,[x_1,x_2]>$ or $I=<[x_1,x_2]>$. One can compute the $\mathcal{M}(H(1 , 0), I)$ when $I=<x_1,[x_1,x_2]>$ or $I=<x_2,[x_1,x_2]>$ similar to Example \ref{e1}. Thus from Example \ref{e2}, the results $(1)$ and $(2)$ follows. 

{\bf Case-2:} If $\dim I=1$, then $I=H^2(m , n)$. By Lemma \ref{lemm9875}, we have $\dim \mathcal{M}(H(m , n), I) =\dim(H(m , n)/H^2(m , n)\otimes H^2(m , n)) =2m+n$.

{\bf Case-3:} If $I\neq H^2(m , n)$. Theorem \ref{the0909} implies $(H(m , n), I)$ is non-capable. By applying Corollaries \ref{corr999} and  \ref{coro852}, we have \\
	\begin{align*}
	\dim \mathcal{M}(H(m , n), I) & =\dim \mathcal{M}(H(m , n)/H^2(m , n),I/(H^2(m , n))-\dim H^2(m , n)\\
	&=(\frac{1}{2}[(k-1)(4m-k)+h(2n-h+1)] \mid 2mn-(2m-k+1)(n-h))-(1\mid0)\\
	&=(\frac{1}{2}[(k-1)(4m-k)+h(2n-h+1)-1]\mid 2mn-(2m-k+1)(n-h))
\end{align*}
\end{proof}

Now we state the same theorem for the Heisenberg Lie superalgebras with odd center and as the proof is very similar we omit it.

\begin{theorem}
Let $J$ be a $(k \mid k+1)$-dimensional graded ideal of $H_m$ and $m \geq 1$. Then 
\begin{enumerate}
\item If $k = 1$, then $\dim \mathcal{M}(H_1, I) = 1~ {\rm{or}} ~2$
\item If $k = 0$, then $\dim \mathcal{M}(H_m, J) = 2m$
\item If $k \geq 1$, then 
\begin{align*}
\dim \mathcal{M}(H_m, J) 
&=(\frac{1}{2}k[(2m-k-1)+(2m-k+1)] \mid m^2-(m-k)^2-1)
\end{align*}
\end{enumerate}
\end{theorem}

\end{document}